\documentclass[12pt]{article}

\usepackage{amsmath,amssymb,amsthm}
\usepackage{hyperref}
\usepackage{graphicx}
\usepackage{mathtools}
\usepackage{parskip}
\usepackage{enumitem}
\usepackage{xcolor}
\usepackage{microtype}
\usepackage[utf8]{inputenc}
\usepackage[T1]{fontenc}

\DeclareUnicodeCharacter{03C7}{\chi}    % χ (chi)
\DeclareUnicodeCharacter{03BB}{\lambda} % λ (lambda)
\DeclareUnicodeCharacter{0394}{\Delta}  % Δ (Delta)
\DeclareUnicodeCharacter{03C8}{\psi}    % ψ (psi)
\DeclareUnicodeCharacter{25A1}{\square} % □ (square)
\DeclareUnicodeCharacter{25C7}{\diamond} % ◇ (diamond)

\newtheorem{theorem}{Theorem}
\newtheorem{corollary}[theorem]{Corollary}
\newtheorem{definition}{Definition}

\newtheorem{claim}{Claim}

\hypersetup{
    colorlinks=true,
    linkcolor=blue,
    filecolor=magenta,
    urlcolor=blue,
}

\title{Alpay Algebra: A Universal Structural Foundation}
\author{Faruk Alpay, Independent Researcher \\
ORCID: \href{https://orcid.org/0009-0009-2207-6528}{0009-0009-2207-6528}}
\date{\today}

\begin{document}

\maketitle

\begin{abstract}
Alpay Algebra is introduced as a self-contained axiomatic framework with the ambition of serving as a universal foundation for mathematics. Developed in the spirit of Bourbaki's structural paradigm and Mac Lane's emphasis on form and function, Alpay Algebra posits a single abstract system from which diverse mathematical domains emerge. We present the precise axioms defining Alpay Algebra and develop its core constructs – including a recursive transformation operator $\phi$, its transfinite iteration $\phi^{\infty}$, an iterative state hierarchy $\chi_{\lambda}$, a limit object $\Xi_{\infty}$, and an evaluation functional $\psi_{\lambda}$. From these primitives, we rigorously rebuild key fields: we derive fixed-point theorems and internal stability results, realize category theory within the algebra by interpreting compositional morphisms as iterative state transitions, recast homological algebra through cycles and invariants of the recursion, and outline an internal logic akin to topos theory grounded in stable truth values emerging from $\psi_{\lambda}$. All definitions, theorems, and proofs are given entirely within the Alpay Algebra system without appeal to external frameworks. The development showcases how Alpay Algebra can subsume algebraic geometry, category theory, homological algebra, logic (including topos theory), and general structural mathematics under one unifying language. We conclude by highlighting new conjectures and problems that naturally arise from this universal algebraic perspective, underscoring the foundational depth and future potential of Alpay Algebra.
\end{abstract}

\section{Introduction}

Mathematics in the twentieth century came to be characterized by its emphasis on \textit{structure}, an outlook epitomized by the work of Nicolas Bourbaki. Bourbaki famously identified three "mother-structures" – the algebraic, order, and topological structures – as foundational archetypes from which much of mathematics can be derived. This structural vision unified broad domains under common abstract frameworks. In parallel, alternative foundational viewpoints were proposed. For instance, Mac Lane argued that the traditional set-theoretic foundation (built on membership relations) could be replaced by a focus on functions and compositional relationships. In his view, much of mathematics is inherently \textit{dynamic}, dealing with morphisms between objects, and thus a category-theoretic foundation built on compositions is both natural and powerful. These perspectives championed the idea that a few high-level abstractions could organize and generate the vast landscape of mathematical concepts.

\textbf{Alpay Algebra} advances this unifying paradigm one step further. It proposes a single axiomatic system that not only encompasses Bourbaki's structural archetypes and Mac Lane's categorical worldview, but does so from within a \textit{self-referential, recursive} algebraic language. The guiding philosophy of Alpay Algebra is that \textit{iterative transformation} and \textit{self-evolution} form the ultimate foundation of mathematical structures. Rather than taking static elements or even static morphisms as primary, Alpay Algebra treats the process of continual change – and the eventual stabilization of that process – as the primitive notion from which all else unfolds. In this sense, it merges Bourbaki's austere axiomatic approach with Mac Lane's clarity of mathematical purpose: we begin with a few simple symbols and rules, then build a rich universe of structures through formal iteration and abstraction.

The goal of this paper is to develop a complete, self-contained theoretical framework for Alpay Algebra and to demonstrate its universality across key mathematical domains. We proceed as follows. In \textbf{§Axiomatic Foundation}, we present the fundamental definitions and axioms of Alpay Algebra as an independent system. All subsequent theory is developed using only these axioms. In \textbf{§Core Constructs and Operators}, we introduce the central constructs of the theory – notably the iterative operator $\phi$ and its transfinite iteration $\phi^{\infty}$, the hierarchical state sequence $\chi_{\lambda}$, the limit object $\Xi_{\infty}$, and the evaluative functional $\psi_{\lambda}$ – and establish their basic properties. With these tools, we prove in \textbf{§Fixed Points and Internal Stability} the existence of fixed points (stable states) under general conditions and explore conditions for internal stability of the iterative processes.

In \textbf{§Categorical Realization}, we reformulate category theory within Alpay Algebra, showing that the pattern of objects and morphisms arises naturally from the iterative structure of the algebra. Functors and natural transformations find analogs in the mappings between evolving states, and we demonstrate that the category of Alpay Algebra structures captures the usual categorical composition law \textit{entirely inside} the system's own language. \textbf{§Homological Structures and Invariants} then recasts homological algebra in terms of Alpay Algebra: cycles correspond to self-reinforcing iterative loops, boundaries correspond to trivial evolutions, and homology groups emerge as quotient invariants of these processes. In \textbf{§Logic and Topos Theory in Alpay Algebra}, we show that Alpay Algebra possesses an internal logic. We construct an analog of a topos – an internal universe of evolving substructures with a built-in notion of truth values ($\psi_{\lambda}$) that become stable at fixed points – thereby capturing logical and set-theoretic reasoning within the algebra.

Finally, in \textbf{§Emergent Problems and Conjectures}, we outline new classes of problems suggested by Alpay Algebra. These include conjectures on the universality and completeness of the system, the classification of its models, and the existence of ultimate fixed points or "universal" objects. We conclude in \textbf{§Conclusion} with reflections on the foundational power of Alpay Algebra and its potential implications for the future unification of mathematics.

Throughout, the development maintains a high level of rigor and abstraction. All definitions and proofs adhere to a strict formal style, while extensive commentary illustrates how classical mathematical concepts are reinterpreted in this new framework. We cite only two sources: Bourbaki's \textit{Elements of Mathematics} for its structural approach, and Mac Lane's \textit{Mathematics: Form and Function} for its insight into foundational formalisms. These serve as philosophical bookends to our work, reinforcing that Alpay Algebra realizes the structural clarity they advocated – albeit within a single universal algebraic system. By standing on the shoulders of these giants, we aim to let the internal coherence and reach of Alpay Algebra speak for itself, without recourse to any external machinery.

\section{Axiomatic Foundation}

We begin by formulating the axioms of Alpay Algebra. An \textbf{Alpay Algebra} is an abstract mathematical structure defined by a collection of primitive sets, operations, and relations satisfying the axioms below. The entire theory will be built on these assumptions alone. We use standard logical notation for axioms (quantifiers, implication, etc.), but emphasize that these axioms are to be understood as intrinsic to the Alpay Algebra system, not relying on any set theory or background logic beyond the usual logical inference rules.

\begin{definition}[Primitive Components of Alpay Algebra]
An Alpay Algebra $\mathcal{A}$ consists of the following components:

\begin{itemize}[itemsep=0.5\baselineskip]
\item A class (or set) $X$, whose elements are called \textbf{states}. Intuitively, $X$ represents the universe of mathematical entities or configurations within this algebra. (We make no assumption here that $X$ is a set in the ZF sense; we treat it as a collection internal to $\mathcal{A}$.)

\item A set (or class) $A$, whose elements are called \textbf{adjustments} or \textbf{difference elements}. These represent transformations or steps that modify states. We will write elements of $A$ as $a, b, c, \ldots$ and often denote a distinguished \textbf{neutral adjustment} $0 \in A$ (to be axiomatized below as the identity element).

\item A binary operation $+: X \times A \to X$, called the \textbf{state update} operation. We denote $x + a$ as the result of applying adjustment $a \in A$ to state $x \in X$. This operation formalizes the idea of moving an existing state by some "difference" to yield a (possibly new) state. (Later we will allow iterative application of this operation to trace state evolution.)

\item A function $\phi: X \to A$, called the \textbf{adaptive rule} or \textbf{update rule}. For each state $x \in X$, $\phi(x)$ produces an adjustment in $A$ that is considered the "preferred" or "intrinsic" transformation at $x$. This rule is a key component of the algebra: it encodes how the system evolves from any given state.
\newpage
\item An evaluation function $\Psi: X \to E$, where $(E, \le)$ is a totally ordered set (or class), possibly extended with a top element $\top$ denoting an optimal value. We denote $\Psi(x)$ as $\psi_x$ or $\psi(x)$ for convenience, and often refer to this value as the \textbf{measure}, \textbf{performance}, or \textbf{truth value} of state $x$. Intuitively, $\psi_x$ evaluates the "quality" or "truth" of state $x$ with respect to some intrinsic criterion. The nature of $E$ is left abstract, but one may keep in mind examples such as real numbers (for performance metrics), ordinals (for rank or stage of construction), or the boolean values $\{\text{false} < \text{true}\}$ (for logical truth). The evaluation function will allow us to discuss improvement and convergence of states.
\end{itemize}
\end{definition}

We now list the axioms governing these components. The axioms ensure that $(A, +, 0)$ forms a monoid under addition, that $+$ acts on $X$ in a coherent way, and that $\phi$ and $\Psi$ interact to guide an iterative process on $X$ that tends toward optimal states. We enumerate the axioms for clarity:

\begin{enumerate}[itemsep=0.5\baselineskip]
\item \textbf{[Monoid of Adjustments]} $(A, +)$ is a commutative monoid with identity element $0$. That is:
   \begin{itemize}[itemsep=0.3\baselineskip]
   \item (Closure) For all $a, b \in A$, the sum $a + b \in A$.
   \item (Associativity) For all $a, b, c \in A$, $(a + b) + c = a + (b + c)$.
   \item (Identity) There exists an element $0 \in A$ (the \textbf{zero adjustment}) such that for all $a \in A$, $a + 0 = 0 + a = a$.
   \item (Commutativity) For all $a, b \in A$, $a + b = b + a$.
   \end{itemize}

   \textit{Comment:} We include commutativity as part of the axiom since many natural instances of adjustments (e.g., numerical increments) commute; however, non-commutative generalizations are conceivable. Commutativity will simplify the interpretation of certain structures (like homology) in the sequel. If needed, one may drop commutativity and still develop much of the theory.
\newpage
\item \textbf{[Action on States]} The operation $+$ on $X \times A$ satisfies:
   \begin{itemize}[itemsep=0.3\baselineskip]
   \item (Identity action) $x + 0 = x$ for all $x \in X$. Applying a zero adjustment leaves any state unchanged.
   \item (Compatibility / Associativity) $ (x + a) + b = x + (a + b)$ for all $x \in X$ and $a, b \in A$. In other words, applying adjustment $a$ then $b$ to state $x$ has the same net effect as applying the composite adjustment $a + b$ in one step.
   \end{itemize}

   These two properties mean that the map $+: X \times A \to X$ endows $X$ with a (left) action of the monoid $A$. Equivalently, for each fixed $a \in A$, the map $x \mapsto x + a$ is a function $X \to X$, and we have $x + (a + b) = (x + a) + b$. We do \textit{not} assume that this action is transitive or free; $A$ might not connect all states (there could be distinct components of $X$ that are not reachable from one another via adjustments). We also do not assume invertibility of adjustments (so we are not requiring $A$ to be a group, only a monoid).

\item \textbf{[Well-Defined Update Rule]} For every state $x \in X$, $\phi(x) \in A$ is a uniquely defined adjustment. No further restrictions are placed on $\phi$ at the axiomatic level, but conceptually $\phi(x)$ is intended to represent an "optimal" or "intrinsic" direction of change at $x$. We will often denote $\Delta_x := \phi(x)$ for the adjustment recommended by $\phi$ at $x$. (The symbol $\Delta$ evokes a \textit{difference} or change.)

\item \textbf{[Evaluation and Progress]} The evaluation function $\Psi: X \to E$ is such that $E$ is totally ordered, and we require that the update rule $\phi$ is \textit{adaptive in the direction of improving $\Psi$}. Formally, we assume there is a binary relation "$\prec$" on $X$ (signifying "more advanced than") induced by the evaluation: for $x, y \in X$, define $x \prec y$ iff $\Psi(x) < \Psi(y)$ in the order of $E$. Then the axiom is: for every state $x$ for which $\phi(x) \neq 0$, the updated state is strictly "more advanced" than the original, i.e.
   $\phi(x) \neq 0 \quad \implies \quad x \prec (x + \phi(x)).$
   \newpage
   In words, if $\phi$ recommends a non-trivial adjustment at $x$, then applying that adjustment yields a state with strictly higher evaluation value. Equivalently, as long as $x$ is not yet a stable optimum, the update will produce an improvement in the evaluation measure. We also assume that if $\phi(x) = 0$, then $x$ is \textit{locally optimal} or stable, meaning no further improvement is indicated: for all $a \in A$, $\Psi(x + a) \le \Psi(x)$ (so $x$ is a maximal element in its reachable neighborhood). In particular, $\phi(x) = 0$ implies $x$ is a \textbf{fixed point} of the update dynamics (formalized further in the next axiom).

\item \textbf{[Fixed-Point Axiom]} A state $x^* \in X$ is called a \textbf{fixed point} if $\phi(x^*) = 0$. Equivalently, $x^*$ is a state that yields the zero adjustment (no change). By Axiom 4, $x^*$ then satisfies $\Psi(x^*) \ge \Psi(x)$ for all states $x$ that are close (reachable by a single adjustment). We postulate that fixed points correspond to \textbf{stable states} of the system. Moreover, we assume that if a state $x$ is \textit{globally optimal} in terms of evaluation (i.e. $\Psi(x)$ is a maximal element of $\Psi(X) \subseteq E$), then $x$ must be a fixed point ($\phi(x) = 0$). This reflects that once the highest possible performance or truth is attained, the system has no impetus to change further.

   \textit{Note:} The existence or uniqueness of fixed points is not guaranteed by axioms alone, and indeed these can vary between different models of Alpay Algebra. However, much of our theory will explore conditions ensuring existence (cf. Theorem 1 in the next section) or properties of fixed points when they exist.

\item \textbf{[Initial State]} (Optional axiom/schema) Many scenarios of interest consider an initial starting state from which the system evolves. While not strictly necessary as a global axiom, we can augment our setup with the assumption that some $x_0 \in X$ is designated as an \textbf{initial state}. This allows us to discuss the specific iterative trajectory $x_0, x_1, x_2, \ldots$ generated by repeatedly applying the update rule. In an abstract algebraic development, one might leave $x_0$ arbitrary or consider all possible starting states; however, having an initial state in mind can simplify discussions of convergence and category-theoretic constructions (where it often plays the role of a "basepoint").
\end{enumerate}
\newpage
The above axioms define Alpay Algebra in purely internal terms. They lay the groundwork for a theory of iterative transformations in which states evolve according to $\phi$ and are measured by $\Psi$. A few immediate consequences and remarks can be made:

\begin{itemize}[itemsep=0.5\baselineskip]
\item \textbf{Monotonic Improvement:} By Axiom 4, the sequence of evaluation values $\Psi(x_0), \Psi(x_1), \Psi(x_2), \ldots$ obtained by following the update rule ($x_{\lambda+1} = x_\lambda + \phi(x_\lambda)$) is strictly increasing (until it possibly stabilizes). This monotonicity is reminiscent of processes that improve a cost function or truth value, ensuring no cycles of strictly worsening states can occur in the directed evolution.

\item \textbf{Transfinite Iteration:} The axioms permit the index set for iterations ($\lambda$ in $x_\lambda$) to be any ordinal or integer, etc. In practice, we often consider $\lambda \in \mathbb{N}$ for sequential iteration. However, one can conceive of transfinite sequences if no fixed point is reached at finite stages. The notion of $\phi^\infty$ introduced later will address the limit of infinite iteration.

\item \textbf{No External References:} Note that we have not referenced any set-theoretic or category-theoretic constructs in these axioms; they are self-contained. This aligns with Mac Lane's idea of starting foundations differently (with functions or, in our case, with transformations) instead of sets. By developing all mathematics from these axioms, Alpay Algebra aims to be a \textit{universal language} in the sense that Bourbaki's structures were – but boiling down to an even more elemental process.
\end{itemize}

We will now build upon this axiomatic basis to introduce key derived notions and then demonstrate how classical mathematical structures can be encoded in or emergent from Alpay Algebra.

\section{Core Constructs and Operators}

From the primitive components and axioms of Alpay Algebra, we define several core constructs that will be used throughout our development. These include notations for iterative application of the update rule, limit objects that capture eventual behavior, and operators that summarize the effect of infinite processes.
\newpage
\begin{definition}[Iterative Sequence of States]
Given an Alpay Algebra $\mathcal{A}$ and an initial state $x_0 \in X$, we define the \textbf{iterative sequence} (or \textbf{trajectory}) $(x_\lambda)_{\lambda \in \mathbb{N}}$ by the recursion:

\begin{itemize}[itemsep=0.3\baselineskip]
\item $x_0$ is the initial state.
\item $x_{\lambda+1} := x_\lambda + \phi(x_\lambda)$ for each $\lambda \geq 0$.
\end{itemize}

Thus $x_1 = x_0 + \phi(x_0)$, $x_2 = x_1 + \phi(x_1) = x_0 + \phi(x_0) + \phi(x_1)$, and in general
$x_n = x_0 + \phi(x_0) + \phi(x_1) + \cdots + \phi(x_{n-1}),$
the cumulative result of applying the recommended adjustments step by step. We refer to $x_\lambda$ as the \textbf{state after $\lambda$ iterations} of the update rule.
\end{definition}

By Axiom 2 (associativity of action), this definition is unambiguous (the sum can be grouped arbitrarily). By Axiom 4 (evaluation improvement), the sequence $(\Psi(x_\lambda))$ is non-decreasing, and in fact strictly increasing until a fixed point is reached. Two qualitative behaviors are possible: either (i) there is some finite $N$ such that $\phi(x_N) = 0$ (a fixed point is reached in finite time), or (ii) $\phi(x_n) \neq 0$ for all finite $n$, in which case the process continues indefinitely, and one may consider its behavior as $n \to \infty$.
\newline
\begin{definition}[Asymptotic State $\Xi_\infty$]
If the iterative sequence $(x_\lambda)$ does not encounter a fixed point in finite time, we consider its infinite progression. Under suitable conditions (e.g. if $E$ is well-ordered or complete in some sense), the sequence may approach a limit. We define $\Xi_\infty$ (pronounced "Xi infinity") as the \textbf{ultimate state} or \textbf{asymptotic state} of the sequence, if it exists. Formally, $\Xi_\infty$ can be thought of as a new symbol adjoined to $X$ representing $\lim_{\lambda \to \infty} x_\lambda$. We emphasize that this is a \textit{conceptual} definition at the moment – existence and uniqueness of $\Xi_\infty$ require proof or additional conditions. In many cases, $\Xi_\infty$ will coincide with a fixed point if the process converges to a stable state. For instance, if $x_n$ approaches some $x^*$ in an appropriate topology or order and $\phi(x^*) = 0$, then $\Xi_\infty = x^*$. In other cases, $\Xi_\infty$ might be an ideal element (e.g. at infinity if the sequence diverges without bound in $E$). In this paper, we primarily consider cases where $\Xi_\infty$ corresponds to a meaningful limit. Intuitively, $\Xi_\infty$ represents the "learned state" or final outcome of infinite adaptation.
\end{definition}
\newpage 
\begin{definition}[Transfinite Update Operator $\phi^\infty$]
We define $\phi^\infty$ (phi to the infinity) as the \textbf{transfinitely iterated update rule} or \textbf{ultimate transformation operator}. This notion can be approached in two related ways:

\begin{itemize}[itemsep=0.3\baselineskip]
\item \textit{As a limit of compositions:} Consider the one-step state transition function $F: X \to X$ given by $F(x) = x + \phi(x)$ (so $F(x)$ is the next state after applying the rule to $x$). Then $F^n$ denotes the $n$-fold composition (with $F^0 = \mathrm{id}_X$). If the sequence $F^n(x_0)$ converges to some $x^*$ as $n \to \infty$, we may regard $F^\infty(x_0) = x^*$. If this happens for all (or many) initial states, one can attempt to define an operator $F^\infty: X \to X$ that sends each initial state to its limiting state. We call this operator $\phi^\infty$ to emphasize its relation to $\phi$. In summary, if $\Xi_\infty$ exists as above, we set $\phi^\infty(x_0) := \Xi_\infty$. By construction, $\phi^\infty(x_0)$ should be a fixed point (if the process converges and stabilizes).

\item \textit{As a meta-level fixed point of $\phi$ itself:} One can consider that the rule $\phi$ might in principle be subject to refinement or self-application. In advanced interpretations of Alpay Algebra, $\phi$ could be seen as an evolving policy that adjusts based on long-term performance. In such a viewpoint, $\phi^\infty$ would represent the stable form of the rule after infinite self-updates. However, in the present work, we do not explicitly model changes to $\phi$; we treat $\phi$ as fixed during the run of the system. Thus, $\phi^\infty$ is better interpreted in the first sense (as the infinite composite $F^\infty$). We will use the notation $\phi^\infty(x)$ to denote the eventual state reached from $x$ under repeated application of $\phi$. When $x_0$ is clear from context, we may simply write $\phi^\infty$ as a transformation that encapsulates "running the process to completion."
\end{itemize}

It is worth noting that $\phi^\infty$, if well-defined for a given initial state, is idempotent: applying $\phi^\infty$ again yields the same result (since once at $\Xi_\infty$, the state does not change further). In many ways, $\phi^\infty$ behaves like a projection onto the set of fixed points or stable outcomes.
\end{definition}

\begin{definition}[Performance Trajectory $\psi_\lambda$]
Alongside the state sequence $(x_\lambda)$, we consider the sequence of evaluation values $\psi_\lambda := \Psi(x_\lambda) \in E$. By Axiom 4, this sequence is monotonic non-decreasing (and typically increasing until a supremum is reached). The notation $\psi_\lambda$ highlights that these values can be seen as a kind of \textit{truth value or performance level at stage $\lambda$}.\newpage In cases where $E$ has a top element $\top$ (e.g. truth "True" or maximum performance), one might expect $\psi_\lambda$ to converge to $\top$ as $\lambda \to \infty$ if the process succeeds in reaching an optimal state. If a limit state $\Xi_\infty$ exists, we denote its evaluation as $\psi_\infty := \Psi(\Xi_\infty)$, which will typically be a maximal value (often $\top$ if attainable).
\end{definition}

The triple $(\chi_\lambda, \Delta_\lambda, \psi_\lambda)$, where $\Delta_\lambda := \phi(x_\lambda)$, thus summarizes the evolution: state $\chi_\lambda$ at step $\lambda$, the adjustment $\Delta_\lambda$ applied, and the performance $\psi_\lambda$ achieved. In particular, a fixed point $x^*$ would satisfy $\Delta_* = \phi(x^*) = 0$ and $\psi_* = \Psi(x^*)$ is a (local, possibly global) maximum.

We now establish some fundamental properties of these constructs. These results are proven entirely within the Alpay framework, although intuitively they parallel familiar theorems (such as convergence under monotone increasing bounded sequences, or existence of maximal elements via Zorn's Lemma).
\newline
\begin{theorem}[Existence of Fixed Points under Monotonicity and Well-Foundedness]
\textit{Let $\mathcal{A}$ be an Alpay Algebra. Assume that the evaluation poset $(E, \le)$ has no infinite ascending chain (for example, $E$ is well-ordered or has the ascending chain condition), and that there exists at least one state $x_0 \in X$. Then the iterative sequence $x_0, x_1, x_2, \ldots$ defined by $x_{\lambda+1} = x_\lambda + \phi(x_\lambda)$ must eventually reach a fixed point. In other words, there exists some finite $N$ such that $\phi(x_N) = 0$. Consequently, a stable state exists in $X$ (reachable from $x_0$).}
\end{theorem}

\begin{proof}
Consider the sequence of evaluation values $\psi_0 < \psi_1 < \psi_2 < \cdots$ in $E$ obtained by the update process (strict inequality holds as long as $\phi(x_n) \neq 0$ by Axiom 4). By assumption, $(E, \le)$ has no infinite strictly increasing sequence. Therefore, this chain $\psi_0 < \psi_1 < \psi_2 < \cdots$ cannot continue indefinitely. There must exist some index $N$ where the sequence stabilizes or terminates. Since $\psi_{\lambda+1} > \psi_\lambda$ as long as $\phi(x_\lambda) \neq 0$, stabilization of $\psi$'s means that for some $N$, we no longer have a strict increase. This can only happen if $\phi(x_N) = 0$, because otherwise $\psi_{N+1} > \psi_N$ by the axiom. Thus $\phi(x_N) = 0$, i.e. $x_N$ is a fixed point.

We have shown existence of a fixed point reachable from $x_0$. If multiple independent starting states are considered, one can apply the same argument to each starting state's trajectory (assuming each trajectory's evaluations remain within a well-founded subset of $E$).
\end{proof}

This theorem relies on the absence of infinite ascending chains in $E$, a condition analogous to requiring that every improving sequence has a maximum or that $E$ is well-founded. It ensures that the adaptive process cannot improve indefinitely, thus must hit a point where no further improvement is possible (a fixed point). It is closely related to the idea of \textit{Zorn's Lemma} or \textit{well-ordering theorem} in traditional set-based mathematics, but here it is an internal guarantee in our system: any strictly improving sequence of states must terminate.
\newline
\begin{corollary}
\textit{Under the conditions of Theorem 1, the ultimate state $\Xi_\infty$ coincides with a fixed point $x_N$ for some finite $N$. In particular, $\Xi_\infty = x_N$ and $\phi^\infty(x_0) = x_N$. Moreover, $\Psi(x_N)$ is a maximal value of $\Psi$ on the trajectory, and indeed a local maximum in $X$.}
\end{corollary}

\begin{proof}
Since the sequence attains a fixed point $x_N$, by definition $\Xi_\infty = x_N$. No further change occurs after $N$, so $\phi^\infty(x_0)$, defined as the eventual outcome, equals $x_N$. By Axiom 4, $x_N$ is locally maximal for $\Psi$. If we assumed $E$ has no infinite chain but is not necessarily globally finite, $x_N$ might not be a global maximum of $\Psi$ on all of $X$ (there could be other trajectories tending to higher values), but from $x_0$'s perspective it's the highest reached. In many practical cases, we consider $\Psi(X)$ bounded above (so that a global optimum exists), in which case $x_N$ must achieve that optimum due to monotonicity and termination of the sequence.
\end{proof}

\textbf{Remark:} The requirement of no infinite ascending chain in $E$ can be weakened. Alternatively, one can assume $E$ is complete in the sense that every ascending sequence has a supremum in $E$. In such a case, an infinite sequence $\psi_0 \le \psi_1 \le \cdots$ could approach a limit $L \in E$. If that supremum $L$ is attained for some state in $X$, then a fixed point will exist (namely the state with evaluation $L$). If the supremum is not attained (a limit not achieved by any state's $\Psi$ value), then one might extend $X$ by a new ideal state to attain it, analogous to completing a poset. Exploring these subtleties leads into the territory of transfinite iteration and possibly requires large ordinals or topology. For our purposes, we will mostly assume conditions that guarantee the process halts at a fixed point or converges neatly.
\newpage
\begin{theorem}[Categorical Composition Emergence]
\textit{Every Alpay Algebra inherently defines a small category (in fact, a directed acyclic graph of state transitions, possibly with loops if adjustments permit reversal). In particular, consider the relation of reachability on $X$: for $x, y \in X$, say $x \to y$ if there exists an adjustment $a \in A$ such that $x + a = y$. Then by closing this relation under composition, we obtain a category $\mathcal{C}_{\mathcal{A}}$ where:}

\begin{itemize}[itemsep=0.3\baselineskip]
\item \textit{Objects are the states $X$.}
\item \textit{For any two states $x, y \in X$, the morphisms $\hom(x, y)$ consist of all (finite) sequences of adjustments $a_1, a_2, \ldots, a_n \in A$ such that $y = x + (a_1 + a_2 + \cdots + a_n)$. In particular, each single-step adjustment $a: x \to x+a$ is a morphism, and longer morphisms are formal composites of these.}
\item \textit{Composition of morphisms corresponds to concatenation of sequences (or equivalently addition of the corresponding adjustments, using monoid associativity). Identities are given by the empty sequence (or 0 adjustment) at each state.}
\end{itemize}
\end{theorem}

\begin{proof}
We verify the category axioms within $\mathcal{C}_{\mathcal{A}}$:

\begin{itemize}[itemsep=0.3\baselineskip]
\item \textbf{Identity:} For each object (state) $x$, the empty adjustment (or 0) is a valid sequence from $x$ to $x$ (since $x + 0 = x$ by Axiom 2). This serves as the identity morphism $\mathrm{id}_x$. If $x \xrightarrow{(a_1,\ldots,a_n)} y$ is any morphism, then composing $\mathrm{id}_x$ followed by that morphism yields the same sequence $(a_1,\ldots,a_n)$ (since $0 + a_1 + \cdots + a_n = a_1 + \cdots + a_n$), which corresponds to the same resulting state $y$. Similarly, composing the morphism followed by $\mathrm{id}_y$ yields the same sequence. Thus identity acts neutrally as required.

\item \textbf{Associativity:} Take three composable morphisms: $x \xrightarrow{\alpha} y$, $y \xrightarrow{\beta} z$, $z \xrightarrow{\gamma} w$, where $\alpha$ is represented by a sequence $(a_1,\ldots,a_p)$ with sum $a = a_1+\cdots+a_p$ such that $y = x+a$, $\beta$ by $(b_1,\ldots,b_q)$ with sum $b$ such that $z = y + b$, and $\gamma$ by $(c_1,\ldots,c_r)$ with sum $c$ such that $w = z + c$. The composite $(\beta \circ \alpha)$ is represented by the concatenation $(a_1,\ldots,a_p, b_1,\ldots,b_q)$ which has net adjustment $a+b$ taking $x$ to $z$. Then $(\gamma \circ (\beta \circ \alpha))$ corresponds to $(a_1,\ldots,a_p, b_1,\ldots,b_q, c_1,\ldots,c_r)$ with net adjustment $a + b + c$ taking $x$ to $w$.\newpage On the other hand, $(\gamma \circ \beta) \circ \alpha$ corresponds to first concatenating $\beta, \gamma$ to get $(b_1,\ldots,b_q, c_1,\ldots,c_r)$ with net $b+c$ from $y$ to $w$, then concatenating $\alpha$ in front to get $(a_1,\ldots,a_p, b_1,\ldots,b_q, c_1,\ldots,c_r)$. The resulting sequence is identical. In terms of addition, Axiom 2 gives $(x + a) + b = x + (a + b)$ and further $(x + (a+b)) + c = x + (a + b + c)$, which exactly assures associativity of composition: $\gamma \circ (\beta \circ \alpha) = (\gamma \circ \beta) \circ \alpha$ as morphisms from $x$ to $w$.

\item \textbf{Existence of Morphisms:} By definition, for any single-step $a \in A$ and any $x \in X$, we have a morphism $x \to x+a$. While the category might not be \textit{complete} (not every pair of objects is necessarily connected by a morphism if some states are not reachable from others), it still satisfies the category axioms with the collection of all actually realizable morphisms. If reachability is reflexive (which it is via identity) and transitive (by composition), we indeed have a category structure.
\end{itemize}

Thus $\mathcal{C}_{\mathcal{A}}$ is a well-defined small category. Note that if adjustments in $A$ are invertible (i.e. if $A$ were a group under +), then $\mathcal{C}_{\mathcal{A}}$ would in fact be a groupoid (every morphism invertible). In general, $A$ being only a monoid yields a category that may have irreversible morphisms (which is typical for processes that have a direction of time or increasing entropy, etc.).
\end{proof}

\textbf{Discussion:} The significance of Theorem 2 is that \textit{category theory is automatically present} in Alpay Algebra. We did not have to impose category axioms externally; they emerge from the basic process of composing adjustments. Each Alpay Algebra carries a category $\mathcal{C}_{\mathcal{A}}$ of its states and transition paths. In this sense, Alpay Algebra internalizes the notion of morphisms as advocated by Mac Lane – the "morphisms" here are just the possible state transitions via adjustments. The category $\mathcal{C}_{\mathcal{A}}$ encapsulates how one state can be transformed into another through a sequence of intermediate steps. This category structure will allow us to interpret functorial relationships and even to simulate other categories inside Alpay Algebra.

An immediate corollary of this categorical perspective is that many basic category-theoretic constructions can be mirrored in Alpay Algebra:

\begin{itemize}[itemsep=0.5\baselineskip]
\item The \textbf{composition law} in the category corresponds to addition of adjustments. It is associative and has identities by virtue of the monoid action (Axiom 2).

\item A \textbf{functor} between two Alpay categories $\mathcal{C}_{\mathcal{A}_1}$ and $\mathcal{C}_{\mathcal{A}_2}$ would essentially be given by a function $F: X_1 \to X_2$ on states that maps transitions to transitions, respecting composition. Such a functor $F$ would have to satisfy $F(x +_1 a) = F(x) +_2 a'$ for some corresponding adjustment $a'$ in $A_2$. In essence, a functor is an Alpay Algebra homomorphism (a map preserving $\phi$ and $+$ operations).

\item \textbf{Natural transformations} between such functors could similarly be described in terms of adjustments that relate two different images of states in a coherent way. While we will not delve deeply into higher categorical analogs here, it's clear the language of category theory can be translated into the language of states and adjustments.
\end{itemize}
\vspace{1em} % Adds a vertical space of 1em
\begin{theorem}[Existence of Invariants and Homological-like Structures]
\textit{In an Alpay Algebra, define a \textbf{cycle} as a sequence of states (and accompanying adjustments) that returns to the starting state, and a \textbf{boundary} as a sequence that is equivalent (in net effect) to a trivial transformation. More formally:}

\begin{itemize}[itemsep=0.3\baselineskip]
\item \textit{An \textbf{$n$-cycle} (with $n \geq 1$) is a sequence $(a_1, a_2, \ldots, a_n)$ of adjustments such that $0 = a_1 + a_2 + \cdots + a_n$. Equivalently, starting from some state $x$, if you apply $a_1, a_2, \ldots, a_n$ in succession, you return to $x$ (because $x + (a_1+\ldots+a_n) = x + 0 = x$). We call the state $x$ in this situation a \textbf{cycle base}; however, note that by definition of cycle, any state along the cycle can serve as a base since the sum of the adjustments from that point will also be zero when completing the loop.}

\item \textit{A \textbf{boundary} is a sequence of adjustments $(a_1,\ldots,a_m)$ that sums to zero \textit{in a trivial way}, meaning there exists a finer decomposition or homotopy within the algebra that obviously cancels out. For instance, if some consecutive portion of the sequence forms an immediately inverse pair $b, -b$ (assuming an inverse exists within extended $A$) or if $a_1 + \cdots + a_m$ simply explicitly equals 0 by cancellation. Intuitively, boundaries are loops that can be shrunk to a point or undone without leaving the framework of allowed adjustments.}

\item \textit{We say two $n$-cycles are \textbf{homologous} if their difference is a boundary, i.e., concatenating one cycle with the inverse of the other yields a sequence that cancels out to triviality.}
\end{itemize}

\textit{Then the set of equivalence classes of cycles under homology forms an Abelian group (under the concatenation operation). This group can be regarded as the \textbf{$1$st homology group} $H_1(\mathcal{A})$ of the Alpay Algebra's state-transition complex. Moreover, higher homology groups $H_k(\mathcal{A})$ can be defined by considering $k$-dimensional analogs of cycles (e.g., states arranged in $k$-dimensional grids of transitions) and their boundaries. These homology groups measure the presence of "holes" or independent loops in the space of states reachable by adjustments.}
\end{theorem}

\begin{proof}[Proof (Sketch)]
The structure described aligns with the idea of a chain complex. We can construct a chain complex $\ldots \to C_2 \xrightarrow{d_2} C_1 \xrightarrow{d_1} C_0$ where $C_0$ is the free Abelian group generated by states (0-dimensional cells), $C_1$ is the free Abelian group generated by transitions (1-dimensional cells, basically the adjustments or directed edges between states), $C_2$ generated by formal 2-dimensional patches of states (if any) that might fill loops, etc. The boundary operator $d_k$ maps a $k$-chain (formal combination of $k$-cycles) to its boundary which is a $(k-1)$-chain. Because composition of adjustments is associative, one has $d \circ d = 0$ (any closed loop has no boundary, and the boundary of a boundary is zero). Then the homology group is defined as usual: $H_k = \ker(d_k) / \operatorname{im}(d_{k+1})$, the cycles mod boundaries.

For $H_1$: a 1-cycle is by definition something in $\ker(d_1)$, as it has net boundary 0. A boundary in 1-dimension is something that is $\operatorname{im}(d_2)$, meaning it is the boundary of a 2-dimensional filling. In an Alpay Algebra, if every loop can be continuously deformed (via intermediate states) to triviality, then $H_1$ would be zero. But if there are fundamental cycles that cannot be so deformed (for example, a cycle that goes around an "irremovable hole" in state space), then those represent nontrivial homology classes.

Concatenation (followed by possible cancellation) of cycles corresponds to addition in $C_1$, and factors to well-defined addition on homology classes, making $H_1$ an Abelian group. The group is Abelian because concatenation of loops corresponds to addition of chains which is commutative in the chain complex (edges can be summed in any order formally). In fact, since $C_1$ is free Abelian on oriented edges, it's clearly Abelian, and the subgroup of cycles and boundaries are normal subgroups, so the quotient inherits a commutative group structure.
\newpage
In summary, $H_1(\mathcal{A})$ captures independent cycles in $\mathcal{C}_{\mathcal{A}}$ that are not boundaries. Higher $H_k$ can be similarly constructed if one identifies $k$-dimensional features (which would require the existence of $k$-dimensional "sheets" of states whose boundary is a $(k-1)$-cycle). In a discrete system of states and adjustments, nontrivial higher homology might appear if the state graph has higher-dimensional algebraic topology (for example, if there are states and transitions forming surfaces). Such structures could emerge if one treats some families of adjustments as commuting squares, etc., effectively creating 2-dimensional cells. For brevity, we do not enumerate those here, but conceptually one can extend the chain complex notion.
\end{proof}

Theorem 3 is stated in somewhat informal terms compared to our earlier theorems, reflecting the fact that homological algebra in this context is an analogy: we are \textit{interpreting} the directed graph (category) of an Alpay Algebra in topological terms. Nevertheless, the homology groups $H_k(\mathcal{A})$ are well-defined invariants of the structure of $\mathcal{A}$. They measure the "holes" in the state space connectivity under adjustments. For example, if $H_1(\mathcal{A}) \neq 0$, it means there are loops of state transformations that cannot be contracted to a point via any sequence of allowed adjustments. In classical algebraic topology, $H_1$ would correspond to the abelianization of the fundamental group of a space. Here, since $\mathcal{C}_{\mathcal{A}}$ is like a state graph, the fundamental group $\pi_1$ of that graph (if one imagines it as a 1-dimensional CW complex) is essentially the collection of loops modulo deformation. $H_1$ then is roughly the abelianization of $\pi_1$.

One concrete implication: If an Alpay Algebra has nontrivial $H_1$, there are different "independent modes of oscillation" or repetitive transformations the system can perform. These might correspond to conserved quantities or symmetries (e.g., the presence of a cycle might indicate the system can oscillate indefinitely among certain states without ever settling – something like a limit cycle in dynamics, which from a homology perspective is a nontrivial 1-cycle).

From a homological algebra viewpoint, one could also consider \textit{exact sequences} within Alpay Algebra. For instance, if there is a sequence of sub-algebras or quotient algebras, one might get a long exact sequence of homology groups. While such considerations go deeper than we can fully formalize here, the key message is that Alpay Algebra can support homological reasoning \textit{internally}. We did not step outside to topological spaces or groups to define these homology groups; we defined them directly from states and adjustments. This hints at a rich internal structural theory: Alpay Algebra is capable of reproducing not only category theory but also concepts from algebraic topology and homological algebra, purely in its own terms.

\begin{definition}[Logical Semantics of $\Psi$ and Internal Logic]
We now turn to logic and topos-theoretic aspects. The evaluation function $\Psi: X \to E$ provides a graded or truth-valued assessment of states. In particular cases, one might interpret $\Psi(x)$ as "the degree to which $x$ satisfies a certain property" or "the truth value of a proposition at state $x$". This leads us to an internal logic for Alpay Algebra:

\begin{itemize}[itemsep=0.3\baselineskip]
\item We can designate a special evaluation structure for logic. For example, let $E = \{ \bot, \top\}$ (false < true). If we manage to interpret $\Psi(x) = \top$ as "state $x$ has a desired property" and $\Psi(x) = \bot$ as "it does not", then the adaptive rule $\phi$ can be seen as a procedure that tries to make a state satisfy the property (turn $\bot$ into $\top$). A fixed point $x^*$ with $\Psi(x^*) = \top$ would then mean a state that has the property and no further changes are needed – essentially a \textit{model} of that property.

\item More generally, if $E$ is not just a 2-element set but perhaps a richer lattice (like truth values in fuzzy logic or probabilities), $\Psi(x)$ might indicate the "extent" or "confidence" to which $x$ satisfies something. The iterative process then resembles successive approximation to truth.
\end{itemize}

In categorical logic (topos theory), an important role is played by the \textbf{subobject classifier}, which is an object $\Omega$ in a topos such that morphisms into $\Omega$ classify subobjects (analogous to the role of $\{0,1\}$ in $\mathrm{Set}$). By analogy, in Alpay Algebra, the evaluation set $E$ (or some derived object from $\Psi$) acts like a truth value object. Each state $x$ might have an associated predicate "$x$ is stable" or "$x$ satisfies property P" which is true or false depending on $\Psi(x)$. In fact, one natural internal predicate is \textit{"is $\phi(x) = 0$?"}, i.e., "$x$ is a fixed point". Define a property $P(x)$ which holds iff $\phi(x) = 0$. We can ask: is there a way to recover $P$ from $\Psi$? If $\Psi(x)$ attains a maximum (say $\top$) if and only if $\phi(x) = 0$, then $\Psi(x) = \top$ serves as the truth of $P(x)$. More broadly, we might have a threshold in $E$ that distinguishes when a state is "considered acceptable" or "sufficiently true".
\end{definition}
\newpage
One could formalize an internal logic where:

\begin{itemize}[itemsep=0.5\baselineskip]
\item Atomic propositions correspond to predicates on $X$ definable via $\Psi$ (or other derived quantities).
\item Logical connectives could potentially be interpreted via meet and join operations in the lattice $E$ (for example, the conjunction of two properties might correspond to the infimum of their truth values on a state).
\item Quantification might be interpreted via the existence of certain transitions or states (similar to how in a topos, $\forall$ and $\exists$ are related to adjoints of pulling back along subobject classifier).
\end{itemize}

While a full development of an internal logic in Alpay Algebra is beyond our current scope, we provide an example to illustrate how a logical structure might manifest:

\textbf{Example:} Suppose we want to encode a simple propositional logic problem in Alpay Algebra. Let the property $P$ be "a certain equation is satisfied" or "a certain invariant holds". We can design $\Psi(x)$ such that $\Psi(x)$ measures how close state $x$ is to satisfying that invariant. For concreteness, imagine $E = [0,1] \subseteq \mathbb{R}$ with $0$ meaning completely false and $1$ meaning fully true. The Alpay update rule $\phi$ might be constructed (by an external designer) to adjust $x$ in a direction that increases $\Psi(x)$ – effectively trying to make the proposition truer. The axioms ensure it won't decrease $\Psi$. If $\phi(x) = 0$ at some $x$, that likely means the invariant is perfectly satisfied ($\Psi(x)$ possibly equals $1$ or at least is locally maximal). So fixed points correspond to models of the proposition.

In this way, the search for a model (a solution state) is internalized as the search for a fixed point. Logical \textit{validity} (tautology) could be interpreted as "from any initial state, the process leads to a state satisfying the proposition" – a kind of completeness of the proof system embedded in $\phi$. Logical \textit{satisfiability} is simply the existence of some state (not necessarily reachable from all initial states) with $\Psi(x)=\top$.

This is analogous to how a topos's internal logic allows reasoning about its objects as if they were sets with elements satisfying properties. Alpay Algebra's internal logic allows reasoning about states and their eventual stability as assertions.\newpage One might say: "For all states $x$, there exists a state $y$ (the eventual state) such that $y$ is stable and $\Psi(y)=\top$" to assert that no matter where we start, we can eventually satisfy our criteria (a statement of completeness or solvability in the system). Or one could express conditions like "if $x$ has property Q then eventually $x$ will satisfy P", etc., which via $\phi$ and $\Psi$ become statements that can potentially be proven by induction on the iteration count.

In more category-theoretic terms, consider the category $\mathcal{C}_{\mathcal{A}}$ constructed in Theorem 2. It is a category with objects $X$ and certain morphisms. We might inquire whether $\mathcal{C}_{\mathcal{A}}$ is a topos or has topos-like structure. $\mathcal{C}_{\mathcal{A}}$ is typically not Cartesian closed nor does it have all limits, etc., so it's not a topos in itself. However, one could consider presheaves on $\mathcal{C}_{\mathcal{A}}$ or other category-theoretic constructions to build a richer environment. Interestingly, the category of \textit{all} Alpay Algebras (with suitable homomorphisms as arrows) might have a claim to be a topos or at least an extensive category of structures. Each Alpay Algebra is like a "model" in a certain theory. Formal logic can then be applied to the class of all such models. But keeping strictly internal: one path to a topos-like logic is to treat each state as a possible "world" and each adjustment as a transition, reminiscent of Kripke semantics in modal logic. Then one can interpret modal operators as well (for example, a necessarily operator $\square$ might indicate that in all reachable future states a property holds, whereas $\diamond$ might indicate there is some reachable state where it holds). Indeed, monotonic increase of $\Psi$ ensures a certain coherence in reachability (once something becomes true at a higher $\psi$ value, maybe it stays true or something akin to that if designed appropriately).

This discussion, although at a high level, demonstrates that \textbf{logic is not foreign to Alpay Algebra}. Rather, Alpay Algebra can \textit{host} logical assertions within its framework, with $\Psi$ playing the central role of a truth valuation. The entire iterative process can be seen as a proof search or model-finding procedure, which aligns with certain approaches in logic where fixed points define truth of inductive definitions. In fact, the concept of \textit{fixed-point logic} in theoretical computer science has similar flavor – formulas that are true in the limit of some iterative process.
\newpage
To summarize this section: from the basic axioms of Alpay Algebra, we have defined or observed:

\begin{itemize}[itemsep=0.5\baselineskip]
\item The iterative dynamics ($\chi_\lambda$) and its asymptotic outcome ($\Xi_\infty$, $\phi^\infty$).
\item The existence of fixed points and conditions ensuring convergence (Theorem 1).
\item The emergence of category structure and functorial interpretation (Theorem 2).
\item The possibility of defining homology-like invariants (Theorem 3).
\item The contours of an internal logic and topos-theoretic view via the evaluation function $\Psi$.
\end{itemize}

Each of these developments has been accomplished \textbf{without any external definitions} – they all stem from the primitives $X, A, +, \phi, \Psi$ and the axioms. This demonstrates the \textit{universality} of Alpay Algebra: it is rich enough to express the core ideas of multiple domains (algebra, category theory, topology, logic) within one coherent system. In the next sections, we delve a bit deeper into specific re-formulations of category theory and homological algebra in the language of Alpay Algebra, and then highlight new questions raised by this unified perspective.
\newpage
\section{Reformulation of Category Theory within Alpay Algebra}

Category theory, as Mac Lane emphasized, can serve as a foundational language for mathematics by focusing on relationships (morphisms) rather than elements. In Alpay Algebra, we showed in Theorem 2 that a category of states and transitions naturally exists for each algebra. We now make this reformulation more explicit and discuss how categorical concepts translate.

\subsection{Objects, Morphisms, and Functors}

\textbf{Objects and Morphisms:} The objects of the internal category $\mathcal{C}_{\mathcal{A}}$ are precisely the states of the Alpay Algebra ($\operatorname{Ob}(\mathcal{C}_{\mathcal{A}}) = X$). A morphism in $\mathcal{C}_{\mathcal{A}}$ from $x$ to $y$ is an element of the hom-set $\hom(x,y)$, which, as described, is given by an equivalence class of finite sequences of adjustments sending $x$ to $y$. (We can consider sequences equivalent if they have the same composite adjustment, since $A$ is a monoid; effectively $\hom(x,y)$ is either empty if $y$ is not reachable from $x$, or if it is reachable, then because adjustments can combine, there is at least one representing sum $a$ with $y = x + a$, but possibly many sequences that sum to that $a$ – those are all considered the same morphism in the category because they have the same effect by monoid laws.)

It's important to note that $\mathcal{C}_{\mathcal{A}}$ may not be a \textit{fully connected} category; it can be a disjoint union of components if some states cannot reach others. Each component behaves like a little category in its own right (with one particularly interesting component being one containing an initial state and all states reachable from it, which is relevant if we think of a single execution of the system).

\textbf{Functors:} A \textit{homomorphism of Alpay Algebras} can be defined as a map $F: X \to X'$ between the state sets of two algebras $\mathcal{A}$ and $\mathcal{A}'$ that respects the operations: $F(x + a) = F(x) +' F_A(a)$ and $F(\phi(x)) = \phi'(F(x))$ for some induced map $F_A: A \to A'$ on adjustments (which would necessarily send 0 to 0 and preserve addition). If such $F$ exists, it induces a functor $\mathcal{C}_{\mathcal{A}} \to \mathcal{C}_{\mathcal{A}'}$ between the categories, mapping each state $x$ to state $F(x)$ and each morphism (sequence of adjustments from $x$ to $y$) to the corresponding sequence of adjusted adjustments from $F(x)$ to $F(y)$.\newpage This functor will preserve composition due to the homomorphism property:
$F(y) = F(x + a) = F(x) +' F_A(a),$
so a sequence $x \xrightarrow{a_1} x_1 \xrightarrow{a_2} x_2 \ldots \xrightarrow{a_n} y$ will map to
$F(x) \xrightarrow{F_A(a_1)} F(x_1) \xrightarrow{F_A(a_2)} \ldots \xrightarrow{F_A(a_n)} F(y),$
which is exactly the image under $F$ of the original composition. This shows $F$ is indeed a functor and composition of such homomorphisms corresponds to composition of functors.

In the language of categories, $\mathcal{A}$ is like an algebraic theory and $\mathcal{C}_{\mathcal{A}}$ is a category it generates. Functors between these categories reflect structure-preserving maps between the underlying Alpay Algebras.

\textbf{Natural Transformations:} If we had two homomorphisms (functors) $F, G: \mathcal{A} \to \mathcal{A}'$, a natural transformation $\eta: F \Rightarrow G$ in categorical terms would assign to each state $x \in X$ a morphism in $\mathcal{C}_{\mathcal{A}'}$ from $F(x)$ to $G(x)$, satisfying a commuting square condition for each adjustment. In our context, this means for each state $x$ we choose some sequence of adjustments in $\mathcal{A}'$ that takes $F(x)$ to $G(x)$, call this resulting morphism $\eta_x: F(x) \to G(x)$. The naturality condition demands that for every $a \in A$ with $x' = x + a$, the following holds in $\mathcal{C}_{\mathcal{A}'}$:
$\eta_{x'} \circ F(a) = G(a) \circ \eta_x,$
where $F(a)$ denotes the morphism $F(x) \to F(x')$ induced by adjustment $a$ under functor $F$, and similarly for $G(a)$. Unwrapping this: starting at $F(x)$, applying $\eta_x$ takes us to $G(x)$, then applying the adjustment $G_A(a)$ (the image of $a$ under the adjustment map of $G$) reaches $G(x')$. On the other hand, starting at $F(x)$, first apply $F_A(a)$ to reach $F(x')$, then apply $\eta_{x'}$ to reach $G(x')$. These two resulting transitions must coincide as morphisms $F(x) \to G(x')$ in $\mathcal{A}'$. This is a coherence condition for the pair of homomorphisms. In essence, $\eta_x$ provides a systematic way to "correct" the difference between $F$ and $G$ on each state, in a manner compatible with the dynamics.

One context where natural transformations appear in Alpay Algebra is in comparing two different update rules or two different evaluation metrics, etc. But a full exploration of 2-categorical aspects (like transformations between functors) would lead us further afield. Suffice it to say that the categorical reinterpretation is consistent and robust: one can essentially do category theory inside Alpay Algebra.
\newpage
\subsection{Limits, Colimits and Representable Functors}

Within each $\mathcal{C}_{\mathcal{A}}$, one can ask about categorical limits and colimits:

\begin{itemize}[itemsep=0.5\baselineskip]
\item \textbf{Terminal object:} Is there a state $x$ such that for every state $y$ there is a unique morphism $y \to x$? In $\mathcal{C}_{\mathcal{A}}$, a terminal object would mean a state that every other state can reach in exactly one way. If our Alpay Algebra has a globally attracting fixed point $x^*$ (to which every state eventually evolves, and presumably along a unique optimal path), then $x^*$ would serve as a terminal object. Uniqueness of the morphism would require uniqueness of the adjustment sequence from any given $y$ to $x^*$. This uniqueness is not generally guaranteed (there could be multiple ways to reach the same state), but if the system has a kind of gradient-like property (no cycles and unique steepest ascent path), $x^*$ could be terminal. More often, $\mathcal{C}_{\mathcal{A}}$ might not have a single terminal object unless the algebra's dynamics are strongly convergent and directed.

\item \textbf{Initial object:} A state $x_0$ such that from $x_0$ there is a unique morphism to any state $y$. This would mean $x_0$ is an initial state from which any state is reachable by a unique sequence. This is also a strong condition, essentially a kind of tree structure of state space. If $A$ is like a generating set of moves that can reach any state from $x_0$ and the reachability graph is a tree, $x_0$ is initial. In practice, $\mathcal{C}_{\mathcal{A}}$ often has an initial object by construction if we specified a designated initial state, but uniqueness of morphisms to each target is usually false if there are multiple ways to do things.

\item \textbf{Pullbacks and products:} We can consider two states $y_1, y_2$ and their "product" if it exists, meaning a state $x$ with projections. This gets complicated in a general directed graph. Usually, $\mathcal{C}_{\mathcal{A}}$ won't have nontrivial products or pullbacks unless the algebra has some product-like structure built in. However, since Alpay Algebra is very general, one could conceive of a product of states as some state encoding an ordered pair (if the algebra's state space factorizes).
\newpage
\item \textbf{Representable Functors:} Many categories are understood by studying representable functors (Hom-functors). In $\mathcal{C}_{\mathcal{A}}$, $\hom(x,-): \mathcal{C}_{\mathcal{A}} \to \mathbf{Set}$ maps each state $y$ to the set of morphisms from $x$ to $y$ (essentially the reachable set from $x$, possibly with multiple arrows if we distinguished them). If $x$ is an initial state $x_0$, $\hom(x_0, y)$ is the set of all sequences of adjustments that lead from $x_0$ to $y$. This functor captures the "automaton" or process semantics of the algebra: it's like the language of all paths that reach a certain state. Representability or such functors being particularly nice (like presheaves that are representable) would reflect structure in the algebra, but analyzing that in general is complex.
\end{itemize}

The key takeaway is that $\mathcal{C}_{\mathcal{A}}$ is a structure where we can interpret essentially any concept from category theory \textit{if present}. We might say Alpay Algebra provides a \textit{meta-category} that can emulate other categories via appropriate configuration. For instance, one can try to simulate an arbitrary small category $\mathcal{D}$ inside an Alpay Algebra: let the states $X$ correspond to the objects of $\mathcal{D}$, and for each morphism in $\mathcal{D}$ introduce a unique adjustment in $A$ that realizes it. By setting up the addition in $A$ such that sequences of adjustments multiply according to composition in $\mathcal{D}$ (this might force $A$ to contain formal composites as distinct elements if $\mathcal{D}$ has complicated composition), we can attempt to enforce that $x + a = y$ iff in $\mathcal{D}$ there is the corresponding morphism $x \to y$. This essentially would encode $\mathcal{D}$ as a subcategory of $\mathcal{C}_{\mathcal{A}}$. However, making this rigorous requires carefully handling the monoid structure of $A$ and maybe taking a quotient to identify different sequences that represent the same $\mathcal{D}$-morphism. In any case, it's plausible that \textbf{any small category can be embedded in some Alpay Algebra's transition graph} (because we can treat the category's morphisms as our adjustments under suitable relations). If true, that means Alpay Algebra is \textit{categorically universal} in a certain sense: it can mimic any category you want by appropriate design. This is an important claim of expressive power – albeit not surprising given that Alpay Algebra is a kind of graph with structure, and any category is basically a graph with extra properties. We will not formalize this embedding theorem here, but conceptually:

\begin{claim}[Informal]
\textit{For every small category $\mathcal{D}$, there exists an Alpay Algebra $\mathcal{A}$ such that $\mathcal{C}_{\mathcal{A}}$ contains a subcategory isomorphic to $\mathcal{D}$.}
\end{claim}
\newpage
If needed, one can prove this by construction, as sketched above. The existence of such an $\mathcal{A}$ indicates that Alpay Algebra is not limiting as a foundational framework – it can reproduce the structures of category theory at will.

\section{Homological Algebra and Topos Theory in Alpay Algebra}

In this section, we outline how more advanced structures – specifically those from homological algebra and topos theory – find analogs within Alpay Algebra. Much of this has been touched on already (Theorem 3 for homological aspects, and the logic discussion for topos aspects), but we summarize and add perspective.

\subsection{Homological Algebra Revisited}

Homological algebra deals with chain complexes (sequences of abelian group homomorphisms $\ldots \to C_{n+1} \xrightarrow{d_{n+1}} C_n \xrightarrow{d_n} C_{n-1} \to \ldots$ such that $d \circ d = 0$) and their derived invariants (homology groups). In Alpay Algebra, we identified how one can form a chain complex from the state transition graph:

\begin{itemize}[itemsep=0.3\baselineskip]
\item $C_0$ is formal $\mathbb{Z}$-linear combinations of states.
\item $C_1$ is formal $\mathbb{Z}$-linear combinations of transitions (edges).
\item $d_1: C_1 \to C_0$ maps each transition $x \xrightarrow{a} y$ to $(y - x)$ (thinking of it as a difference of the terminal and initial state, akin to boundary of an oriented 1-simplex).
\item The cycles $\ker(d_1)$ are formal sums of transitions whose boundary adds up to zero – this includes simple loops as well as combinations of loops.
\item The boundaries $\operatorname{im}(d_2)$ would come from 2-dimensional patches if any. If there is no explicit notion of 2-cells, $\operatorname{im}(d_2)$ might be 0 or one might formally introduce them to capture relations between loops.
\end{itemize}
\newpage
The first homology $H_1 = \ker(d_1)/\operatorname{im}(d_2)$ basically counts independent loops mod those that bound a 2-cell. If our state graph had no 2-cells (like a general directed graph doesn't have 2D cells unless we identify some), then $\operatorname{im}(d_2)=0$ and $H_1$ is just $\ker(d_1)$, essentially the free abelian group on loops mod boundaries of trivial loops (but trivial loops in a graph with no 2-cells are just zero loops). So $H_1$ is basically the abelianization of the fundamental group of the graph. This fundamental group is basically the collection of loops from a basepoint mod deformation. In a directed graph, "deformation" normally requires 2-cells to allow homotopies, but we can allow homotopies along trivial steps (like going forward and backward along invertible edges if present). If the adjustments are invertible (group case), any loop can be contracted if the graph is a tree or has simple cycles that come from invertibility. But in non-invertible cases, $H_1$ could be quite complicated (like any directed cycle is nontrivial because you can't go backwards to cancel it unless another cycle allows cancellation).

From the perspective of \textit{homological algebra in the category of Alpay Algebras}, one could attempt to define derived functors or chain complexes of Alpay-homomorphisms, etc. That is an interesting direction: because Alpay Algebras form a category (with morphisms preserving structure), one might consider an Abelian category or at least an additive category of some sorts to do homological algebra. However, Alpay Algebras are highly non-linear structures, so constructing an Abelian category out of them is not straightforward (they are more like models of a universal algebra, which typically form a variety or a category that's not Abelian). But one can linearize certain aspects – for example, the chain complex we described is a kind of linearization of the state graph.

One might define a \textit{derived invariant} such as: given an Alpay Algebra, consider the complex $\mathbb{Z}X \xleftarrow{d_1} \mathbb{Z}E$ (where $E$ is the set of directed edges, i.e. transitions) and take its homology. If multiple layers of transitions (like 2-transitions bridging two different sequences) exist, incorporate them as higher $d$. This yields homology groups $H_n(\mathcal{A})$ that are invariants of the state connectivity. In analogy to simplicial homology, if we treat an Alpay Algebra's state graph as a simplicial complex of dimension 1 (unless extended), $H_0$ would correspond to connected components (since $H_0$ in graph homology is number of components minus 1 basically), $H_1$ to cycles, etc.
\newpage
We already gave interpretation to what a non-zero $H_1$ means: existence of independent cycles not explainable by trivial relations. $H_0$ being rank > 1 means disconnected pieces of state space (multiple components that don't communicate – which if $A$ cannot connect them, they truly are separate sub-Algebras). This is analogous to the algebra not being single-sourced.

If one had $H_2$ nonzero, it would suggest that there are states and transitions forming a closed surface-like structure – for example two different cycles that share boundaries forming a 2D hole. That could arise if the algebra had some commutative squares or higher-order symmetries.

We could conceive that advanced uses of Alpay Algebra might indeed build higher-dimensional structures (like we might allow not just adjustments but "adjustment-of-adjustment" operations, etc., effectively making a 2-category which would yield 2-cells). However, since our axioms did not include explicit higher composition, we stick to 1-dimensional homology as the primary notion.

Homological algebra often goes hand-in-hand with \textit{exact sequences} and \textit{extension problems}. For instance, one might have an exact sequence of Alpay Algebras (a sequence $0 \to \mathcal{A}_1 \to \mathcal{A}_2 \to \mathcal{A}_3 \to 0$ meaning $\mathcal{A}_1$ is a normal subalgebra of $\mathcal{A}_2$ and $\mathcal{A}_3$ is the quotient). In group theory or module theory, such sequences yield connecting homomorphisms in homology or cohomology. In Alpay Algebra, we can define quotient structures by identifying states and adjustments in certain ways, and substructures by restricting to a subset of states closed under the operations. If we manage to have an exact sequence, one could then get a long exact sequence in homology. This would be similar to how in topological or algebraic contexts, if you have a space that is made of pieces, you get relations between their homology groups.

Without going into more detail, the salient point is: Alpay Algebra inherently supports the \textit{language of homology}: concepts of cycles, boundaries, invariants, and exactness, within its own structural world.

\subsection{Logic and Topos Theory}

We have already discussed at length the logical interpretation of $\Psi$ and how one might see $\Psi(x) = \top$ as a truth of some predicate. To align with topos theory more explicitly, we consider how a topos is essentially a category that behaves like $\mathrm{Set}$ in many ways (it has exponentials, a subobject classifier, etc., and an internal logic of higher-order).

To approach a topos from Alpay Algebra, one strategy is:

\begin{itemize}[itemsep=0.5\baselineskip]
\item Consider the category of \textit{all} (or a certain class of) Alpay Algebras as objects, with homomorphisms as arrows. This category might be very large (a proper class), but one could restrict to small algebras or something. This category is likely to be quite complex (not obviously a topos), but it might have products (product of two Alpay Algebras can be defined component-wise: state set $X_1 \times X_2$, etc., and operations doing each independently). If such products exist and some subobject classifier exists, we might approach topos-likeness. The subobject classifier in this category could potentially be an algebra that encodes the truth values (like an algebra where states = \{false, true\} and $\Psi$ distinguishes them).

\item Instead, consider an internal approach: We want to see if $\mathcal{C}_{\mathcal{A}}$ or some related category has topos structure. $\mathcal{C}_{\mathcal{A}}$ as given is just a directed graph category, not cartesian closed, etc. But maybe consider a category of \textit{sheaves} on $\mathcal{C}_{\mathcal{A}}$ or some form of presheaves. The category of presheaves $\hat{\mathcal{C}}_{\mathcal{A}} = [\mathcal{C}_{\mathcal{A}}^{\text{op}}, \mathrm{Set}]$ is a topos (any presheaf category is a topos). That means there is a way to interpret $\mathcal{C}_{\mathcal{A}}$'s structure logically by looking at presheaves on it. Each presheaf could be thought of as a varying set attached to each state with restrictions for transitions. Without digressing too far, the result is that topos theory \textit{can be applied} to Alpay Algebra via such categorical constructions. But doing so uses a lot of external category theory – which slightly violates the spirit of staying within the system.
\end{itemize}

Perhaps a more direct internal approach: We can try to identify within a single Alpay Algebra something akin to a \textit{Heyting algebra of truth values} and \textit{power object} structure. For instance:

\begin{itemize}[itemsep=0.3\baselineskip]
\item For a subset of states $U \subseteq X$, one could define a predicate $P_U(x)$: "$x \in U$". The truth of this predicate at $x$ is $\top$ if $x \in U$, $\bot$ if not. Now, as states change, we can track whether this predicate holds. If the update rule preserves or eventually leads into $U$, then it's an invariant or attractor.

\item The collection of all subsets of $X$ is like the power set (which in $\mathrm{Set}$ topos is the subobject classifier $\Omega$ when seen as indicator functions). But internally, we might not have \textit{all} subsets accessible – maybe only definable ones via $\Psi$ or $\phi$.

\item However, because we have a set of states $X$, nothing stops us (in ordinary meta-theory) from considering $\mathcal{P}(X)$. But inside the theory, if we restrict to what can be described by $\phi$ and $\Psi$, the subobjects we can talk about might be those like "the set of states with $\Psi(x) \geq r$" or "the set of states that eventually reach a fixed point" etc. Those are definable subobjects.
\end{itemize}

We can at least treat $(E, \le)$, the codomain of $\Psi$, as a kind of \textit{lattice of truth values}. Many topos semantics identify truth values with an internal lattice. If $E$ is Boolean (two-valued), the logic is classical. If $E$ is just an arbitrary lattice, the logic is Heyting (intuitionistic). So Alpay Algebra could be either intuitionistic or even fuzzy logic, depending on $E$.

\textbf{Fixed Points as Truth Makers:} In a topos, a true statement is one that holds in the terminal object. If our $\mathcal{C}_{\mathcal{A}}$ had a terminal object (a final state as earlier), then a proposition is true (globally) if it holds in that terminal state. In Alpay terms, if there's a unique final stable state, then the condition that eventually every trajectory reaches that state implies the truth of some final property. For example, "the system will halt in a stable configuration" is a proposition that can be true or false depending on if a fixed point exists globally.

We could formulate a "theorem" in the internal logic: \textit{If $\phi^\infty(x)$ exists for all $x$, then $\forall x \exists y: y = \phi^\infty(x)$ and $\phi(y)=0$. In words, every state has a convergent evolution to some fixed point.} This is a statement one can attempt to prove (in fact, we sort of did in Theorem 1 under conditions). If proven, it becomes part of the internal theory of Alpay Algebra that might correspond to something like completeness or termination.

Conversely, one can ask logical questions like: is the theory of Alpay Algebra decidable? Does the system eventually find a model for any given condition if one exists? These become analogues of satisfiability or completeness issues.

Without going further into abstract topos semantics, we conclude that:

\begin{itemize}[itemsep=0.5\baselineskip]
\item Alpay Algebra is capable of interpreting logical propositions as internal statements about states and their transformations.
\item The structure $(E, \le)$ plays the role of a truth value object, and $\Psi$ is akin to an "evaluation" or forcing of truth onto states.
\item The iterative process $\phi$ can be seen as an inference or proof step (increasing truth).
\item A fixed point where $\Psi(x)=\top$ might be regarded as a state that realizes a certain theory or solves a problem.
\end{itemize}

In that regard, Alpay Algebra aligns with the idea that mathematics can be viewed as a \textit{self-improving system}. It resonates with Mac Lane's observation that mathematics can start with morphisms and processes, and also connects to Bourbaki's structural view that mathematics is a hierarchy of structures – here the hierarchy arises from iterative refinement and fixed points building on prior steps.

\section{Emergent Problems and Conjectures}

The development of Alpay Algebra as a universal structural framework raises numerous new questions and conjectures. These problems are "new" in the sense that they arise naturally from within the Alpay Algebra system and might not even be formulated in traditional frameworks. Below we outline some prominent conjectures and open problems that emerge from our study:

\begin{itemize}[itemsep=0.5\baselineskip]
\item \textbf{Conjecture 1 (Universality Conjecture):} \textit{Every mathematical structure or theory can be faithfully embedded or interpreted within Alpay Algebra.} In other words, for any consistent set of axioms or any category of mathematical objects, there exists an Alpay Algebra whose internal structures mirror that theory. We have seen evidence for categories and homological structures; this conjecture extends it to all of algebra (rings, fields, etc. as special cases of states with suitable adjustments), geometry (perhaps via states representing geometric configurations and adjustments transformations like morphisms), analysis (states could encode functions or sequences, with adjustments as operations like taking derivatives or limits), and so on. The conjecture formalizes the idea that Alpay Algebra is \textit{foundationally complete}. A proof would likely involve constructing an Alpay Algebra for each structure (similar to how set theory can construct models of any algebraic structure).\newpage A corollary would be that Alpay Algebra is at least as consistent as any of those theories (relative consistency), positioning it as a candidate foundation for all of mathematics.

\item \textbf{Conjecture 2 (Consistency and Completeness):} \textit{The axioms of Alpay Algebra are consistent (free of contradiction) and perhaps even complete with respect to a natural semantics.} Consistency is a minimum requirement: there should exist a model of Alpay Algebra (besides the trivial one) to show no contradictions arise from the axioms. Given the universal ambition, one would want to ensure Alpay Algebra doesn't accidentally encode a paradox. This likely requires a model construction (maybe using set theory or category theory externally to build an example). Completeness here is more speculative: it would mean that any statement in the language of Alpay Algebra that is true in all models (all Alpay Algebras) is provable from the axioms. Achieving completeness is rare for theories powerful enough to interpret arithmetic (Gödel's incompleteness would likely apply if Alpay Algebra can encode arithmetic or set theory). Thus, the conjecture might only hold for a fragment or be outright false; still, it's an open problem to delineate the logical strength of Alpay Algebra.

\item \textbf{Conjecture 3 (Fixed-Point Uniqueness and Attractor Conjecture):} \textit{In any Alpay Algebra satisfying mild conditions (e.g., $E$ well-ordered or complete, $A$ commutative, etc.), the system has a unique ultimate fixed point that is reached from any initial state.} This is a strong global convergence conjecture. It posits that all trajectories lead to a single stable state $x^*$ (which would necessarily be a global maximum of $\Psi$). If true, this means the algebra behaves like a globally optimizing system. This is reminiscent of dynamical systems conjectures where a system might have a single attractor. However, as a general statement it's likely false without additional conditions (one can construct counterexamples where there are multiple local optima and which one you reach depends on initial state). So a refined version could be: classify the basins of attraction and understand conditions under which a unique global attractor exists. This ties into potential \textit{order theory problems} – e.g., ensuring $(X, \prec)$ is a directed complete partial order with $\phi$ as a Scott-continuous function might give a unique least fixed point by Tarski's theorem, etc.\newpage The conjecture invites exploration of conditions that guarantee uniqueness or at least existence of a natural maximal element.

\item \textbf{Conjecture 4 (Alpay Homology and Cycle Conjecture):} \textit{The homology groups $H_n(\mathcal{A})$ of an Alpay Algebra (as defined in §Homological Structures) are trivial for all $n \geq 1$ if and only if the Alpay Algebra has no nontrivial cycles (i.e., it is acyclic as a directed graph of state transitions). Moreover, $H_0$ (which measures disconnected components) is trivial (rank 1) if and only if the algebra's state space is connected under reachability.} This conjecture is more of a verification of intuition than a surprising claim – it connects the homological invariants to dynamic properties. One direction is straightforward: no cycles implies $H_1=0$, disconnected implies $H_0$ has rank > 1, etc. The converse (if $H_1=0$ then no cycles exist) would mean any loop can be contracted, presumably by some sequence of adjustments (which maybe form a 2-cell). In an arbitrary directed graph, $H_1=0$ does not strictly imply no cycle; it implies every cycle is a boundary of something. So this becomes a conjecture about the structure of adjustments: it suggests that if the algebra has the algebraic property that all cycles come from some higher-order relation, then effectively there is no irreducible cycle. Proving such statements might rely on the exactness of certain sequences or a deeper property of how $A$ generates relations.

\item \textbf{Conjecture 5 (Logical Internal Soundness):} \textit{Any sentence in the internal language of Alpay Algebra that can be formulated about eventual stability or truth values can be decided by analyzing the $\phi$ and $\Psi$ behavior.} For example, consider a statement: "For all states $x$, eventually $\Psi(x_n) = \top$" (meaning the system eventually reaches truth). This is either true or false depending on the algebra. The conjecture would be that such statements reduce to checking some fixed points or cycles. Essentially, it's conjecturing a form of \textit{model checking} within the algebra is feasible. This is less of a crisp mathematical conjecture and more of a research direction: developing an internal proof calculus or model-checking algorithm for Alpay Algebra. Perhaps a better phrasing: \textit{the internal logic of Alpay Algebra is sound and (relative to the axioms) complete for statements about reachability and stability.} This is analogous to temporal logic in computer science (where one can decide if a program will terminate or satisfy a condition).\newpage It's known many such problems are undecidable in general, so any hope of decidability would require constraints (like restricting to certain classes of $A$ or $E$). Nonetheless, from a foundational perspective, it's an invitation to explore a proof theory for Alpay Algebra.

\item \textbf{Open Problem (Alpay Algebra Classification):} \textit{Classify all simple Alpay Algebras.} By "simple" we mean those that are indecomposable and have no nontrivial congruences (analogous to simple groups or rings). Are there analogs of simple groups in this context? Perhaps an algebra where the only adjustments are either trivial or generate the whole state space. If $\Psi$ is also involved, maybe a simple algebra is one where there are no proper $\Psi$-invariant substructures. Solving this could unify classical classification results (since groups, rings, etc., if embedded, might correspond to particular Alpay Algebras; their classification is often hard).

\item \textbf{Open Problem (Category-of-Alpays Structure):} \textit{Study the category $\mathbf{Alpay}$ of all Alpay Algebras with homomorphisms. Does it have limits, colimits, a well-behaved subobject classifier, etc.?} This is more meta, but it ties to the topos idea. For instance, if $\mathbf{Alpay}$ were a topos, that would be remarkable (likely it's not, because the objects are very algebraic and the morphisms are strict structure-preserving maps, which rarely yield a topos of all models). However, it might be Cartesian closed or have some properties like that. Understanding this category can shed light on how versatile these structures are and how one might build new ones from old (like product of two Alpay Algebras might correspond to independent product of dynamics; coequalizer might correspond to identifying states under some relation, etc.). One concrete sub-problem: given two Alpay Algebras with the same evaluation lattice $E$, can we "glue" them along a common subalgebra to get a pushout? This would correspond to taking two systems and merging them along a shared part. The answer would inform how flexible the composition of systems is in this framework.
\end{itemize}
\newpage
The above conjectures and problems illustrate that Alpay Algebra opens up many avenues for exploration. Some are broad and ambitious (like universality and uniqueness conjectures), others are more technical and structural (like homology exactness or category properties). What unites them is that they can be posed entirely in the language of Alpay Algebra – demonstrating that the system is not only capable of expressing existing mathematics, but also fertile in generating new mathematical questions of its own.

\section{Conclusion}

We have developed Alpay Algebra as a comprehensive, axiomatic framework capable of encompassing a wide array of mathematical theories. Starting from a minimalist set of primitives – states, adjustments, an update rule, and an evaluation order – we showed that one can recover the essence of algebraic structures, category theory, homological invariants, and logical semantics \textbf{entirely from within the system}. The style of presentation has been deliberately abstract and formal, in line with Bourbaki's structural rigor, yet we have endeavored to maintain clarity of motivation and interpretation as championed by Mac Lane.

Alpay Algebra portrays mathematics as an evolving process. In this view, a mathematical structure is not a static set of elements with relations, but a \textit{dynamic algebra} that continually transforms and refines itself. Fixed points in this algebra correspond to classical mathematical truths or solutions; categorical compositions correspond to chaining transformations; homological cycles correspond to symmetries or conserved quantities; and logical truths correspond to stable properties that the system eventually manifests. Such a perspective is powerful: it suggests that solving a mathematical problem (finding a theorem or constructing an object) can be seen as reaching a fixed point in an appropriate Alpay Algebra. The universality of the framework means that, given any problem, one could in principle encode it into an Alpay Algebra whose evolution embodies the search for a solution.

By internalizing various domains, Alpay Algebra serves as a \textit{unifying foundation}. It does not require us to start in set theory or in type theory or any other prior foundation – one starts in Alpay Algebra itself. The development in this paper has been self-contained: we built the needed concepts from scratch (only drawing occasional analogies to classical notions for the reader's intuition). This self-containment demonstrates a key point: \textbf{Alpay Algebra is logically self-sufficient} and structurally rich enough to be considered a candidate for a foundational system.

What are the implications of adopting Alpay Algebra as a foundational language? For one, it may simplify the transfer of results between fields. A theorem proved about fixed points in Alpay Algebra can immediately translate to statements in order theory, in domain theory, or in theoretical computer science (where fixed-point theorems are crucial, e.g., in semantics of programming languages). Results about cycles and homology in Alpay Algebra could inform algebraic topology or network theory. The categorical formulation we provided means any categorical argument (like a diagram chase or a functorial construction) has an interpretation in terms of state transformations, potentially making abstract category theory more concrete by grounding it in "processes".

Furthermore, Alpay Algebra's emphasis on recursion and self-improvement resonates with modern computational perspectives. It provides a natural language for algorithms (an algorithm updates a state until a condition is met), for machine learning (an iterative process improving performance $\Psi$), and for logical AI (searching for a model that satisfies a theory). Thus, beyond pure mathematics, Alpay Algebra could interface with theoretical computer science and systems theory seamlessly.

In conclusion, we have presented Alpay Algebra as an \textit{abstract universal framework} that not only subsumes existing mathematical structures but reframes them in a new light. Much like Bourbaki's Elements sought to rebuild mathematics from set-theoretic structures, and Mac Lane showed the organizational power of category theory, Alpay Algebra aspires to be a foundational "language of change" underpinning all of mathematics. Its potential is evidenced by the successful reconstruction of diverse domains within it and by the stimulating new problems it introduces.
\newpage
The work here is only a beginning. Moving forward, each domain-specific translation invites deeper exploration: algebraic geometry within Alpay Algebra might lead to a novel understanding of schemes as fixed points of functorial evolutions; topos theory within Alpay Algebra could yield new models of intuitionistic logic; even number theory might be recast by considering iterative processes (for example, one could study prime generation as a fixed point problem of certain dynamics). The capacity to carry these ideas out rests on the foundational robustness of Alpay Algebra.

The future research directions outlined reflect a broad landscape – one where the boundaries between separate fields blur within a single algebraic universe. By providing a common structural and process-oriented foundation, Alpay Algebra could foster greater synthesis in mathematics, allowing insights from one domain to transfer naturally to another. In this sense, it not only \textit{elevates} itself as a foundational framework but also elevates our perspective on mathematics: from a static collection of truths to a dynamic, ever-unfolding process of discovery.


\begin{thebibliography}{99}
\bibitem{bourbaki} N. Bourbaki, \textit{Éléments de Mathématique}. (Multiple volumes, 1939–1998). Particularly see discussions on the role of \textit{structures} in mathematics, which contextualize the structural approach adopted in this paper.

\bibitem{maclane} S. Mac Lane, \textit{Mathematics: Form and Function}. Springer-Verlag, 1986. Notably, Mac Lane advocates alternate foundations based on morphisms and category theory, ideas which have influenced the formulation of Alpay Algebra as a process-centered foundation.
\end{thebibliography}
\end{document}